\documentclass{amsproc}
\usepackage{mathrsfs}
\usepackage{bbm}
\usepackage{tikz}
\usepackage{amsfonts}
\usepackage{amssymb,amsmath,amscd,amsthm}
\usepackage{hyperref}
\usepackage{color}

\usepackage[all]{xy}

\DeclareFontFamily{U}{matha}{\hyphenchar\font45}
\DeclareFontShape{U}{matha}{m}{n}{
	<5> <6> <7> <8> <9> <10> gen * matha
	<10.95> matha10 <12> <14.4> <17.28> <20.74> <24.88> matha12
}{}
\DeclareSymbolFont{matha}{U}{matha}{m}{n}

\DeclareMathSymbol{\Lt}{3}{matha}{"CE}
\DeclareMathSymbol{\Gt}{3}{matha}{"CF}

\newcommand{\delete}[1]{}
\newcommand{\overbar}[1]{\mkern 1mu\overline{\mkern-2mu#1\mkern-1mu}\mkern 1mu}

\newcommand{\bqa}{\begin{equation}}
\newcommand{\eqa}{\end{equation}}
\newcommand{\bea}{\begin{eqnarray}}
\newcommand{\eea}{\end{eqnarray}}
\newcommand{\bna}{\begin{eqnarray*}}
\newcommand{\ena}{\end{eqnarray*}}
\newcommand{\bma}{\begin{pmatrix}}
\newcommand{\ema}{\end{pmatrix}}

\def\valpha{\text{\scalebox{0.88}[1.02]{$\alpha$}}} 
\def\vepsilon{\text{\scalebox{0.9}[1]{$\varepsilon$}}} 

\def\sumx{\sideset{}{^\star}\sum}

\def\bz{{\mathbb Z}}
\def\br{{\mathbb R}}

\def\sl2z{SL(2,\bz)}
\def\psl2z{PSL(2,\bz)}

\def\gl2r{GL(2,\br)}

\def\Q{\mathbb{Q}}

\def\sym{\mathrm {sym}}
\def\ssharp{\scalebox{0.6}{$\sharp$}}
\def\splus{\scalebox{0.66}{$+$}}
\def\sminus{\scalebox{0.66}{$-$}}
\def\spm{\scalebox{0.66}{$\pm$}}

\newtheorem{lemma}{Lemma}[section]
\newtheorem{thm}[lemma]{Theorem}

\newtheorem{prop}[lemma]{Proposition}
\newtheorem{conj}[lemma]{Conjecture}
\theoremstyle{definition}

\newtheorem{remark}[lemma]{Remark}

\newtheorem*{acknowledgement}{Acknowledgements}

\newcommand{\bit}{\begin{itemize}}
\newcommand{\eit}{\end{itemize}}

\title{Bias of Root Numbers for Modular Newforms of Cubic Level}

\author[ Q. Pi  and Z. Qi]{Qinghua Pi and Zhi Qi}

\address{School of Mathematics and Statistics,
Shandong Univeristy, Weihai,
Weihai 264209,
China}
\email{qhpi@sdu.edu.cn}

\address{School of Mathematical Sciences, Zhejiang University, Hangzhou, 310027, China}
\email{zhi.qi@zju.edu.cn}

\subjclass[2000]{11F11, 11F72.}

\keywords{bias of root numbers, modular newforms, simple Petersson formulas}

\date{\today}
\thanks{The first named author is supported by
China Postdoctoral Science Foundation (Grant No. 2018M632658), Shandong Provincial Natural Science Foundation (Grant No. ZR2020MA001), and  in part by Innovative Research Team in University (Grant No. IRT16R43). The second named author is supported by the National Natural Science Foundation of China (Grant No. 12071420).}

\begin{document}
\begin{abstract}
Let $H^{\spm}_{2k} (N^3)$ denote the set of modular newforms of cubic level $N^3$, weight $2 k$, and root number $\pm 1$. For $N > 1$ squarefree and $k>1$, we use an analytic method to establish neat and explicit formulas for the difference $|H^{\splus}_{2k} (N^3)| - |H^{\sminus}_{2k} (N^3)|$ as a multiple of the product of $\varphi (N)$ and the class number of $\mathbb{Q}(\sqrt{- N})$. In particular, the formulas exhibit a strict bias towards the root number $+1$. Our main tool is a root-number weighted simple Petersson formula for such newforms.
\end{abstract}
\maketitle

\section{Introduction}
\setcounter{equation}{0}

Let $S_{2k} (M)$ denote the space of modular forms of level $M$, weight $2 k$, and trivial nebentypus. According to the Atkin--Lehner theory of newforms (see \cite{Atkin-Lehner-1} and \cite[\S 6.6]{Iwaniec-Topics}), $S_{2k} (M)$ has an orthogonal decomposition with respect to the Petersson inner product
\begin{align*}
	S_{2k} (M) = S_{2k}^{\flat} (M) \oplus S_{2k}^{\star} (M),
\end{align*}
where $ S_{2k}^{\flat} (M) $ is the space of oldforms and $S_{2k}^{\star} (M)$ is the   space spanned by newforms. A newform $f \in S_{2k}^{\star} (M)$ is an eigenfunction of all the Hecke operators $T_n$, normalized so that the first Fourier coefficient is $1$. Let  $\lambda_f (n)$ denote its $n$-th Hecke eigenvalue, which is known to be real.   Then we have the Fourier expansion
\begin{align*}
	f (z) = \sum_{n=1}^{\infty} \lambda_f (n) n^{k-1/2} e (n z), \quad \mathrm{Im} (z) > 0,
\end{align*}
where $e (z) = e^{2\pi i z}$.
Let $H^{\star}_{2k} (M)$ denote the orthogonal basis for $S_{2k}^{\star} (M)$ consisting of   newforms of level $M$ and weight $2k$. The dimension of $ S_{2k}^{\star} (M) $ (or the cardinality $|H^{\star}_{2k} (M)|$) is well known by the work of G. Martin (see \cite[Theorem 1]{Martin-Dimension}).

For $f \in H^{\star}_{2k} (M)$, let $\epsilon_f \in \{ \pm 1 \}$ denote the root number of $f$, i.e.  the sign in the functional equation of its $L$-function $L (s, f)$. 
We would like to consider the splitting    $$H^{\star}_{2k} (M) = H^{\splus}_{2k} (M) \cup H^{\sminus}_{2k} (M)$$ according to the sign of root number. A natural question is: What can we say about $|H^{\spm}_{2k} (M)|$ or, in other words, the distribution of root numbers in $H^{\star}_{2k} (M)$?

When the level $M > 1$ is {\it squarefree},\footnote{The case $M =1$ is not of interest here, for the root number is $i^{2k}= (-1)^k$, so either $H^{\star}_{2k}(1) = H^{\splus}_{2k} (1)$ or $H^{\star}_{2k}(1) = H^{\sminus}_{2k} (1)$.} Iwaniec, Luo, and Sanark  established the Petersson formulas over $H^{\star}_{2k} (M)$ and $H^{\spm}_{2k} (M) $  and used them to prove the asymptotic formula (see \cite[Corollary 2.14]{ILS-LLZ}):
\begin{align}\label{1eq: ILS, squarefree}
	|H^{\spm}_{2k} (M)| = \frac {2k-1} {24} \varphi (M) + O \big((kM)^{5/6}\big),
\end{align}
where as usual $\varphi (M)$ is
Euler's totient function. In particular, this implies that the root numbers are equidistributed between $+1$ and $-1$ as $kM \rightarrow \infty$. Note that the formula of G. Martin (see \cite[Theorem 1]{Martin-Dimension} or \cite[Theorem 2.1]{Martin-Bias}) in this case reads:
\begin{align}\label{1eq: |H star|}
 |H^{\star}_{2k} (M)| = \frac {2k-1} {12} \varphi (M) +  \bigg(\frac 1 4 + \left\lfloor \frac {k} 2 \right\rfloor - \frac {k} 2 \bigg)  v_2^{\ssharp} (M)  +   \bigg(\frac 1 3 + \left\lfloor \frac {2k} 3 \right\rfloor - \frac {2k} 3 \bigg)v_3^{\ssharp} (M) + \delta (k, 1) \mu (M),
\end{align}
where $v_2^{\ssharp}  $ and $v_3^{\ssharp} $ are multiplicative functions defined by
\begin{align*}
  v_2^{\ssharp} (p) = \chi_{-4} (p) - 1, \quad v_3^{\ssharp} (p) = \chi_{-3} (p) - 1,
\end{align*} 
$\delta ({k, 1})$ is the Kronecker $\delta$-symbol that detects $k = 1$,   $\mu  $ is the M\"obius function, and $\chi_{-4}$ and $\chi_{-3}$ are the quadratic characters attached to $\mathbb{Q}(\sqrt{-4})$ and  $\mathbb{Q}(\sqrt{-3})$, respectively.

Recently, using Yamauchi's trace formula for  Atkin--Lehner   involutions on $S^{\star}_{2k} (M)$ in \cite{Yamauchi-trace}\footnote{Some mistakes in \cite{Yamauchi-trace} were corrected by Skoruppa and Zagier \cite{SZ-trace}.}, K. Martin (\cite[\S 2]{Martin-Bias}) obtained the following formula for squarefree $M > 3$ (the formula for prime $M > 3$ was essentially contained in  \cite{Wakatsuki-prime-level}, in which Yamauchi's trace formula was also used):
\begin{align}
	|H^{\splus}_{2k} (M)| - |H^{\sminus}_{2k} (M)| = c_M h (D_{-M}) - \delta ({k, 1}),
\end{align}
where
\begin{equation}\label{1eq: defn of cM}
c_M = \left\{\begin{aligned}
&  1 / 2, && \text{ if } M   \equiv 1, 2 \, (\mathrm{mod}\, 4), \\
& 1, && \text{ if } M \equiv 7 \, (\mathrm{mod}\, 8), \\
& 2, && \text{ if } M \equiv 3 \, (\mathrm{mod}\, 8),
\end{aligned}\right.
\end{equation}
$D_{-M}$ and  $ h (D_{-M}) $, respectively, denote the discriminant and the class number of the imaginary quadratic field $\mathbb{Q} (\sqrt{- M}) $. He also proved that for $M = 2$ or $3$,
\begin{equation}
	|H^{\splus}_{2k} (2)| - |H^{\sminus}_{2k} (2)| =
	\left\{ \begin{aligned}
		&1,    && \text{ if } k\equiv 0,1 \, (\mathrm{mod} \ 4), \\
		&0,   && \text{ if } k\equiv 2, 3 \, (\mathrm{mod} \ 4),
	\end{aligned}\right.
\end{equation}
and
\begin{equation}
	|H^{\splus}_{2k} (3)| - |H^{\sminus}_{2k} (3)| =
	\left\{ \begin{aligned}
		&1,    && \text{ if } k\equiv 0,1 \, (\mathrm{mod} \ 3), \\
		&0,   && \text{ if } k\equiv 2 \, (\mathrm{mod} \ 3).
	\end{aligned}\right.
\end{equation} 
It is therefore clear that, in any fixed $H^{\star}_{2k} (M)$ (with squarefree $M > 1$), there is a {\it strict bias} of root numbers towards $+ 1$ with magnitude on the order of the class number $h (D_{-M})$.
Moreover, in view of the dimension formula for  $|H^{\star}_{2k} (M)|$ as in \eqref{1eq: |H star|}, K. Martin immediately obtained explicit formulas for $ |H^{\spm}_{2k} (M)| $. For a fundamental discriminant $D < 0$, we recall Dirichlet's class number formula
\begin{equation}\label{1eq: Dirichlet class}
L (1, \chi_D) = \frac {2\pi h (D)} {w (D) \sqrt{|D|}},
\end{equation}
and the well-known upper bound $L (1, \chi_{D }) \Lt \log |D| $ (as usual $\chi_{D }$ is the quadratic character attached to $ \mathbb{Q} (\sqrt{D}) $, $L(s, \chi_{D})$ is its Dirichlet $L$-function, $h (D)$ is the class number, and $w (D)$ is the number of units). From these we infer that $h (D_{-M}) = O (\sqrt{M} \log M)$.\footnote{Note that we also have Siegel's ineffective lower bound $h (D_{-M}) \Gt_{\vepsilon} M^{1/2 - \vepsilon}$ for any $\vepsilon > 0$.} Consequently,  the error term in \eqref{1eq: ILS, squarefree} may be improved into $O (\sqrt{M} \log M)$.\footnote{By simple considerations, one may find that the error term $O \big(2^{\omega (M)}\big)$ in \cite{Martin-Bias} is not correct ($\omega (M)$ is the number of prime factors in $M$). Actually, one has $ |H^{\star}_{2k} (M)| =   {(2k-1)}  \varphi (M) / 12 + O \big(2^{\omega (M)}\big) $ for $M$ squarefree (see \eqref{1eq: |H star|}).}

K. Martin said  in \cite{Martin-Bias} that ``the trace formula of Yamauchi is valid for arbitrary level, but becomes considerably more complicated", and that, ``in principle", his approach also works for non-squarefree level, ``but the resulting formulas may be messy".

\subsection*{Main results}

In the present paper, we use a new analytic method due to Balkanova, Frolenkov \cite{Ba-Fr-Mean}, and Shenhui Liu  \cite{Liu-First-Moment} to derive a still very ``neat" formula for $|H^{\splus}_{2k} (M)| - |H^{\sminus}_{2k} (M)|$ in the case of {\it cubic} level $M = N^3$ with $N > 1$  squarefree. As indicated in \cite{Martin-Bias}, this would possibly be
prohibitively difficult from the approach via Yamauchi's trace formula.  Instead, our approach is based on the  root-number weighted ($\Delta^{\ssharp}$-type) {\it simple} Petersson formula over $H^{\star}_{2k} (N^3)$ established in \cite{PWZ2019} (see \eqref{2eq: Petersson sharp} in Proposition \ref{prop-sPTF}).


The following is our main theorem.

\begin{thm}\label{thm-dimension-formula}
	Let notation be as above. Assume that $N > 1$ is squarefree and   $k > 1$. When $N > 3$, we have
	\begin{align}
		|H^{\splus}_{2k} (N^3)| - |H^{\sminus}_{2k} (N^3)| =
 {c_N \varphi(N) h (D_{-N})}  ,
	\end{align}
	where the constant $c_N$ is  given as in {\rm\eqref{1eq: defn of cM}}, $\varphi (N)$ is Euler's totient function, and $h(D_{-N})$ is the class number of  $\Q(\sqrt{-N})$.
	When $N=2$, we have
	\begin{equation}
	|H^{\splus}_{2k} (8)| - |H^{\sminus}_{2k} (8)| =
	\left\{ \begin{aligned}
	&0,    && \text{ if } k\equiv 0,1 \, (\mathrm{mod} \ 4), \\
	&1,   && \text{ if } k\equiv 2, 3 \, (\mathrm{mod} \ 4).
	\end{aligned}\right.
	\end{equation}
	When $N=3$, we have
	\begin{equation}
		|H^{\splus}_{2k} (27)| - |H^{\sminus}_{2k} (27)| =
		\left\{ \begin{aligned}
			&1,    && \text{ if } k\equiv 0,1 \, (\mathrm{mod} \ 3), \\
			&2,   && \text{ if } k\equiv 2 \, (\mathrm{mod} \ 3).
		\end{aligned}\right.
	\end{equation}
\end{thm}

This theorem manifests that the root numbers in $H^{\star}_{2k} (N^3)$ have a {\it strict bias}  in favor of $+ 1$ with magnitude on the order of  $ \varphi(N) h (D_{-N})$. Also note that there is a perfect equidistribution of root numbers only when $N=2$ and $k \equiv 0, 1 (\mathrm{mod}\, 4)$. 

According to \cite[Theorem 1]{Martin-Dimension}, 
for $N>1$ squarefree, we have
\begin{align}
	|H^{\star}_{2k} (N^3) | = \frac {2k-1} {12} {\varphi(N)^2 \nu (N)}  +  \bigg(\frac 1 4 + \left\lfloor \frac {k} 2 \right\rfloor - \frac {k} 2 \bigg) \delta (N, 2) +  \bigg(\frac 1 3 + \left\lfloor \frac {2k} 3 \right\rfloor - \frac {2k} 3 \bigg) \delta (N, 3),
\end{align}
with
\begin{align*}
	\nu (N) = N \prod_{p\mid N} \bigg(1+\frac 1 p \bigg).
\end{align*}
Consequently,
\begin{align}
	|H^{\spm}_{2k} (N^3)| = \frac {2k-1} {24} {\varphi(N)^2 \nu (N)} \pm \frac{c_N}{2}  \varphi(N) h (D_{-N})
\end{align}
if $N > 3$, 
\begin{equation}
|H^{\spm}_{2k} (8)|
= \frac 1 2 \left\lfloor \frac k 2 \right\rfloor  + \left\{ \begin{aligned}
&  0,  && \text{ if }  k\equiv 0, 1 \, (\mathrm{mod} \ 4), \\
&  \pm \frac 1 2,  && \text{ if }  k\equiv 2, 3 \, (\mathrm{mod} \ 4), 
\end{aligned} \right.
\end{equation}
and
\begin{equation}
|H^{\spm}_{2k} (27)|
= \frac 1 2 \left\lfloor \frac {8k} 3 \right\rfloor - \frac 1 2 + \left\{ \begin{aligned}
&  \pm \frac 1 2  ,  && \text{ if }  k\equiv 0, 1 \, (\mathrm{mod} \ 3), \\ 
& \pm 1,  && \text{ if }  k\equiv 2 \, (\mathrm{mod} \ 3).
\end{aligned} \right.
\end{equation}
From these we deduce the asymptotic formula
\begin{align}
 	|H^{\spm}_{2k} (N^3)| = \frac {2k-1} {24} {\varphi(N)^2 \nu (N)} + O \big(N^{3/2} \log N \big).
\end{align}

We conclude this Introduction with the following conjecture.

\begin{conj}
	For any $M > 1$ we have $|H^{\splus}_{2k} (M)| \geqslant |H^{\sminus}_{2k} (M)|$, and  $|H^{\splus}_{2k} (M)| - |H^{\sminus}_{2k} (M)|$ is independent on $k$ if $M$ is not divisible by $2$ or $3$ and $k > 1$. Moreover, we have the asymptotic formula
	\begin{align}
		|H^{\spm}_{2k} (M)| = \frac {2k-1} {24} M s_0^{\ssharp} (M) + O (\sqrt{M} \log M),
	\end{align}
	in which $ s_0^{\ssharp}  $  is the multiplicative function in Definition {\rm $1'$ (A)} in {\rm\cite{Martin-Dimension}},  satisfying
	\begin{equation}
		s_0^{\ssharp} (p) = 1-\frac 1 p, \quad s_0^{\ssharp} (p^2) = 1-\frac 1 p - \frac 1 {p^2}, \quad s_0^{\ssharp} (p^\valpha) = \bigg(1-\frac 1 p \bigg) \bigg(1- \frac 1 {p^2} \bigg) \quad \text{\rm($\valpha \geqslant 3$)}.
	\end{equation}
\end{conj}


\begin{acknowledgement}
	We would like to thank Yongxiao Lin for helpful discussions and careful readings of the manuscript. We also thank the anonymous referees for their long lists of comments. 
\end{acknowledgement}

\section{Preliminaries}
\setcounter{equation}{0}

\subsection{Notation}

Let $\varphi (n)$ be the Euler   totient function, and $\mu (n)$ be the M\"obius function. For complex $s$, define
\begin{align}\label{2eq: sigma (n)}
\sigma_{s}(n)=\sum_{d\mid n} d^{s}.
\end{align}
Let $\delta (m, n)$ be the Kronecker  $\delta$-symbol that detects $m = n$.

Set $e (z) = e^{2\pi i z}$. We define the   Kloosterman sum
\begin{align*}
S (m, n; c) = \sumx_{ a (\mathrm{mod}\, c) }  e \bigg(  \frac {a m + \overbar{a} n } c \bigg),
\end{align*}
where  {\small $\displaystyle \sumx$} restricts the summation to the primitive residue classes, and $\overbar{a}$ denotes the multiplicative inverse of $a$ modulo $c$.  We record here Weil's bound for the Kloosterman sum
\begin{align}\label{2eq: Weil}
|S (m, n; c)| \leqslant (m, n, c)^{1/2} c^{1/2} \sigma_0 (c).
\end{align}

Let $J_{\nu} (z)$ be the Bessel function of the first kind (\cite{Watson}). Henceforth, we shall be only concerned with $J_{2k-1} (x)$ for real $x > 0$ and integer $k \geqslant 1$. We have the following estimates  (see \cite[\S \S 2.11, 7.1]{Watson})
\begin{equation}\label{2eq: bounds for Bessel}
J_{2k-1} (x) \Lt_{ k } \left\{ \begin{aligned}
& x^{2k-1} , \ &&   x \leqslant 1, \\
& x^{-1/2} , \ && x > 1.
\end{aligned} \right.
\end{equation}
Moreover, the following Mellin--Barnes integral representation for the Bessel function will be used later (see \cite[\S 3.6.3]{MO-Formulas})
\begin{align}\label{2eq: Mellin-Barnes}
J_{2k-1} (x) =\frac{1}{4\pi i}\int_{(\sigma)}
\frac{\Gamma\left(k - \frac 1 2 + \frac 1 2 s \right)}
{\Gamma\left(k + \frac 1 2   - \frac 1 2 s\right)} \Big(\frac x 2 \Big)^{-s}ds,
\end{align}
for $1-2k < \sigma < 1$; by Stirling's formula, the integral is absolutely convergent if  $1-2k < \sigma < 0$.

For  $0 < \valpha \leqslant 1$, we define the Hurwitz zeta function
\begin{align}\label{2eq: Hurwitz}
\zeta (s, \valpha) = \sum_{n=0}^{\infty} \frac 1 {(n+\valpha)^s}, \quad \mathrm{Re} (s) > 1.
\end{align}
We recollect several basic results on  $\zeta (s, \valpha)$ from Theorems 12.4, 12.6, and 12.23 in \cite{Apostol}.  First of all, $\zeta (s, \valpha)$ has an analytic continuation to the whole complex plane except for a simple pole at $s = 1$ with residue $1$. We have Hurwitz's formula
\begin{align}\label{2eq: Hurwitz formula}
\zeta (1-s, \valpha) = \frac {\Gamma (s)} {(2\pi)^s} \big(e (-s/4) F (\valpha, s) + e (s/4) F (- \valpha, s) \big), \quad \mathrm{Re}(s) > 1,
\end{align}
with
\begin{align}\label{2eq: defn of F(b, s)}
F (\beta, s) = \sum_{n=1}^{\infty} \frac {e(n\beta)} {n^{s}} ,
\end{align}
where $\beta$ is real and $\mathrm{Re}(s) > 1$.  Moreover, we have the following crude bound: 
\begin{equation}\label{2eq: bounds for zeta}
\zeta (s, \valpha) - \valpha^{-s} \Lt | \mathrm{Im}(s) |^{3/2 - \lfloor \mathrm{Re}(s) + 1/2 \rfloor } 
\end{equation}
for $\mathrm{Re}(s) < 3/2$ and $ | \mathrm{Im}(s) | > 1 $.

\subsection{Simple Petersson formulas over  $H^{\star}_{2k} (N^3)$}

We  introduce the symmetric square $L$-function for $f \in H^{\star}_{2k} (N^3)$:
\begin{align}
	L(s, \sym^2(f))=\zeta^{(N)}(2s)\mathop {\sum_{ (n,N)=1}} \frac{\lambda_f(n^2)}{n^s},\quad \mathrm{Re}(s)>1,\label{eq-symmetric-square}
\end{align}
where $\lambda_f(n)$ are the Hecke eigenvalues of $f$, and $\zeta^{(N)}(s)$ is the partial Riemann zeta-function defined by
\begin{align*}
	\zeta^{(N)}(s) = \prod_{p \hskip 1pt \nmid N} \frac 1 {1 - p^{-s}}, \quad \mathrm{Re}(s)>1.
\end{align*}
It is known  that $L(s, \sym^2(f))$ admits a holomorphic continuation to the entire complex plane.

We can now state the two types of simple Petersson formulas in \cite{PWZ2019}.

\begin{prop}\label{prop-sPTF}
	Assume that $N > 1$ is  squarefree  and $k> 1$. Set
	\begin{align}\label{2eq: defn Delta *}
		\Delta_{2k, N^3}^{\star}(m,n   ) & = \sum_{f\in H^{\star}_{2k} (N^3)}\frac{\lambda_f(m)\lambda_f(n   )}{L(1, \sym^2(f) )},
	\end{align}
	and
	\begin{align} \label{2eq: defn Delta sharp}
	\Delta_{2k, N^3}^{\ssharp}(m,n   ) = \sum_{f\in H^{\star}_{2k} (N^3)} \epsilon_f \frac{\lambda_f(m)\lambda_f(n   )}{L(1, \sym^2(f))}.
	\end{align}
	For  $(mn   ,N)=1$, we have
	\begin{equation}\label{2eq: Petersson *}
		\begin{split}
			\Delta_{2k, N^3}^{\star}(m,n   )= & \, \delta ({ m,n   }) \frac{(2k-1)N^2\varphi(N)}{2\pi^2}\\
			&  +\frac{(-1)^k(2k-1)}{\pi}
			\sum_{c= 1}^{\infty} \frac{A_{N}(c)}{c} S(m,n   ;N^2c) J_{2k-1}\left(\frac{4\pi\sqrt{mn   }}{N^2c}\right)
			,
		\end{split}
	\end{equation}
	and
	\begin{align}\label{2eq: Petersson sharp}
		\Delta_{2k, N^3}^{\ssharp}(m,n   )
		=&\frac{(2k-1)N^{3/2}}{\pi}
		\sum_{ (c,N)=1}\frac{S\big(\overbar{ N}{}^3m,n   ;c\big)}{c}
		J_{2k-1}\left(\frac{4\pi\sqrt{mn   }}{N^{3/2}c}\right),
	\end{align}
	where
	$A_{N}(c)=\prod_{p\mid N}A_{p}(c)$ with
	\begin{equation*}
		A_{p}(c)=\left\{ \begin{aligned}
			&-1,  \ && \text{ if } p\nmid c,\\
			&p-1,  \ && \text{ if } p\mid c,
		\end{aligned}\right.
	\end{equation*}
	and $ \overbar{ N}$ is the multiplicative inverse of $N$ modulo $c$. 
\end{prop}

The meaning of the adjective ``simple" is threefold. First, the formulas are closely related to the {\it simple} trace formula of Deligne and Kazhdan in the early 1980's, in which a supercuspidal matrix coefficient is used (see \cite{Gelbart-AStrace}).  Second, for each prime $p | N$, the local component $\pi_p$ associated to $f \in H^{\star}_{2k} (N^3)$ is known to be a {\it simple} supercuspidal
representation  (see \cite{G-R-simple,Simple-GL(n)-1-KL}). Third,    the above Petersson formulas of cubic level $M=N^3$ look quite {\it simple} compared to those of squarefree level $M > 1$  as in \cite{ILS-LLZ}. 

The $\Delta^{\star}$-type Petersson formula  in \cite{ILS-LLZ} has been generalized in a similar manner to the case of arbitrary level by  Ng, Nelson, and Barrett et. al. \cite{Ng-Petersson,Nelson-Petersson, Miller-1-level-density} (see also \cite{ILS-LLZ,Rouymi,Bl-Mi-Second-Moment}). The $\Delta^{\star}$-type formula \eqref{2eq: Petersson *} is deducible from theirs after simplifications, but the $\Delta^{\ssharp}$-type formula in \eqref{2eq: Petersson sharp} is novel.

Finally, we  comment that 
the condition $k> 1$ is required in the proof in \cite{PWZ2019} so that a certain matrix coefficient of $\pi_\infty$ is integrable.  With additional efforts, their proof might still be carried out for $k=1$.

\subsection{A Mellin--Barnes type integral and Gauss' hypergeometric function}

As in \cite[(5-2)]{Ba-Fr-Mean} and \cite[\S 2.5]{Liu-First-Moment}, for  $1/2 < \mathrm{Re} (s) < 2k$, we define the integral
\begin{equation}\label{I-z-x}
	I_{k, s}(x)=
	\frac{1}{2\pi i}\int_{(\sigma )}
	\frac{\Gamma\left(k - \frac 1 2 + \frac 1 2 w \right)}
	{\Gamma\left(k + \frac 1 2   - \frac 1 2 w\right)}
	\Gamma({1-s-w})
	\sin \Big( \pi \frac {s+w} 2 \Big)  x^w d w, \quad  \text{($x > 0$)},
\end{equation}
with $ 1-2k< \sigma  < 1-\mathrm{Re} (s) $. We have the following explicit formulas for $I_{k, s}(x)$ involving Gauss' hypergeometric function. See \cite[Lemmas 5.2--5.4]{Ba-Fr-Mean} and \cite[Lemma 2]{Liu-First-Moment} (the formulation of the latter is simpler but equivalent to that of the former via some identities for the gamma function).

\begin{lemma}\label{Lemma-Gauss-hyper}
Let $x > 0$. Assume that $1/2 < \mathrm{Re}(s) < 2k $.	We have
\begin{equation}
I_{k, s}(x) = \left\{ \begin{aligned}
& {2 (-1)^k} \cos \Big(\frac{\pi s}{2} \Big)
\frac{\Gamma (2k-s)}{\Gamma(2k)} x^{1-2k}
{_2F_1}\bigg(k-\frac{s}{2},
k+\frac{1}{2}-\frac{s}{2};
2k;\frac{4}{x^2}\bigg), && \text{ if } x > 2, \\
& \frac{ (-1)^k 2^s}{\sqrt{\pi}}\cos \Big(\frac{\pi s}{2} \Big)
\frac {\Gamma  (2k -s ) \Gamma  \big(s- \frac{1} {2}  \big)} {\Gamma (2k+s-1)}, && \text{ if } x = 2, \\
& \frac{\Gamma\big(k -   \frac 1 2 s \big)}
{\Gamma \big(k  + \frac 1 2 s\big)}  x^{1-s}
{_2F_1} \bigg(k-\frac{s}{2},1-k-\frac{s}{2}; \frac{1}{2};\frac{x^2}{4} \bigg) , && \text{ if } x < 2,
\end{aligned} \right.
\end{equation}
in which ${_2F_1} (\valpha, \beta; \gamma; z)$ is Gauss' hypergeometric function. 

\end{lemma}

Recall that for $|z| < 1$, ${_2F_1} (\valpha, \beta; \gamma; z)$ is defined by Gauss' hypergeometric series
\begin{align}\label{2eq: Gauss hypergeometric}
{_2F_1} (\valpha, \beta; \gamma; z) = \frac {\Gamma (\gamma)} {\Gamma (\valpha) \Gamma (\beta)} \sum_{n=0}^{\infty} \frac {\Gamma (\valpha+n) \Gamma (\beta+n)} {\Gamma (\gamma+n) n!} z^n .
\end{align}
It is known (see \cite[\S 2.1]{MO-Formulas}) that
\begin{equation}\label{2eq: Gauss evaluated}
 {_2F_1} \Big(\frac{\valpha}{2},-\frac{\valpha}{2}; \frac{1}{2};\sin^2 \theta \Big)=\cos(\valpha \theta).
\end{equation}

\subsection{Zagier's $L$-functions}
For integers $c $ and $\varDelta $, with $c \geqslant 1$, let
$$
\rho (c, \varDelta)=\# \big\{ a \, (\mathrm{mod} \, 2 c) : a^2\equiv \varDelta \, (\mathrm{mod} \, 4 c) \big\}.
$$
Clearly, $\rho (c, \varDelta) =0$ if $\varDelta \equiv 2,3 \, (\mathrm{mod} \, 4)$.
For $\mathrm{Re} (s) > 1$, define the $L$-function
\begin{equation}
  L(s, \varDelta)= \frac{\zeta(2 s)}{\zeta(s)} \sum_{c= 1}^{\infty} \frac {\rho (c, \varDelta)}{c^s}.
\end{equation}


We recollect some fundamental results of Zagier \cite[Proposition 3]{Zagier} as follows.

\begin{lemma}\label{lemma-Zagier}

The $L$-function $ L(s, \varDelta)$ has an analytic continuation to the whole complex plane, which is entire except for a simple pole at  $s=1$ when $\varDelta$ is a square.
More precisely, $ L(s,\varDelta)$ can be expressed in terms of standard Dirichlet series{\rm:}
	\begin{equation}\label{2eq: L(s, Delta) = ...}
 L(s, \varDelta) = \left\{\begin{aligned}
& 0, && \text{ if } \varDelta \equiv 2,3 \, (\mathrm{mod} \, 4), \\
& \zeta (2s-1), && \text{ if } \varDelta = 0, \\
& L(s,\chi_D)
 \sum_{d| q} \mu(d) \chi_D(d) \frac
 {\sigma_{1-2s}(q/d)} {d^s} , &&  \text{ if } \varDelta \equiv 0,1 \, (\mathrm{mod} \, 4), \varDelta \neq 0,
\end{aligned} \right.
	\end{equation}
where in the last case we have written $\varDelta=D q^2$ with $D$  the discriminant of $\Q(\sqrt{\varDelta})$,  $\chi_{D}  = \left(\hskip -0.7pt \text{\Large $\frac{D}{\cdot}$}\hskip -0.7pt\right)$ is the Kronecker symbol, $L(s,\chi_D)$ is the Dirichlet $L$-function of $\chi_D,$ and $\sigma_{1-2s}$ is defined as in {\rm\eqref{2eq: sigma (n)}}.
	
\delete{	
	{\rm(3)}  For $\varDelta<0$, the value of $ L (s, \varDelta)$ at $s = 1$ is given by\footnote{The formula in Proposition 3 (iv) in \cite{Zagier} should read:
	\begin{align*}
	 L (1, \varDelta) = \frac {\pi} {\sqrt{|\varDelta|}} L(0, \varDelta) = \frac {\pi} {\sqrt{|\varDelta|}} H (|\varDelta|).
	\end{align*}}
	\begin{equation}
	\label{2eq: L(1, Delta)}
	L(1,\varDelta)=\frac{2\pi}{\sqrt{|\varDelta|}}
	\mathop{  \sum_{  d^2| \hskip 1pt |\varDelta|} }_{\varDelta/d^2\equiv 0,1 \, (\mathrm{mod}\, 4)} \frac{h(\varDelta/d^2)}{w(\varDelta/d^2)}.
	\end{equation}
}
	
\end{lemma}

\delete{
\begin{prop}[Props 3.4.5 and 3.5.1 in \cite{BuVo2007}]
For $\varDelta\equiv 0,1\bmod 4$,
$\rho(q;\varDelta)$ is a multiplicative function in $q$. Moreover,
\bit
\item for $p\nmid \varDelta$, and $e\geqslant 1$, we have
\bna
\rho(p^e;\varDelta)=1+\chi_{\varDelta}(p),
\ena
\item for $p\mid\varDelta$, $p\nmid f$ and $e\geqslant 1$,
\bna
\rho(p^e;\varDelta)=\left\{
\begin{aligned}
&1,\quad &&e=1\\
&0,\quad &&e\geqslant 2
\end{aligned}
\right.
\ena
\eit
\end{prop}
}

\begin{lemma}\label{Dirichlet-series-involving-Kloosterman-sum}
Let $N, m, n$ be   integers, with $N \geqslant 1$.
For $\mathrm{Re} (s) > 3/2$, we have
\begin{equation}\label{3eq: Kloosterman = L}
\sum_{(c, N) = 1} \frac{1}{c^{1+s}}\sum_{a (\mathrm{mod}\, c)}
S(m,a^2;c)e\Big( \frac{a n}{c}\Big)=\frac{ L^{(N)}(s,n^2-4m)}{\zeta^{(N)}(2 s)},
\end{equation}
with
\begin{align}
L^{(N)} (s, \varDelta) = \frac{\zeta^{(N)}(2 s)}{\zeta^{(N)}(s)} \sum_{(c, N) = 1} \frac {\rho (c, \varDelta)}{c^s}.
\end{align}
\end{lemma}

\begin{proof}
According to the proof of Lemma 4.1 in \cite{Ba-Fr-Mean}, we have
\begin{equation*}
\sum_{a (\mathrm{mod}\, c)}
S(m,a^2;c) e\Big( \frac{a n}{c}\Big) = c\sum_{d | c} \mu(d) \rho (c/d, n^2-4m),
\end{equation*}
and \eqref{3eq: Kloosterman = L} follows immediately upon using the formula for a product of two Dirichlet series.
\end{proof}

\begin{remark}\label{rem: L(N)(s) = L(s)}
	For the last case in {\rm\eqref{2eq: L(s, Delta) = ...}}, a simple observation is that $L^{(N)} (s, \varDelta) = L  (s, \varDelta)$ if $N | D$ and $(N, q) = 1$. This is because $ L(s,\chi_D) = L^{(N)}(s,\chi_D) $ for $ N |D$,  while the whole \underline{}sum over divisors $d | q$ is contained  in  $L^{(N)} (s, \varDelta)$ for $(N, q) = 1$. 
\end{remark}

Finally, we have the following trivial bound for $ L  (s, \varDelta) $ in the $\varDelta$-aspect.

\begin{lemma}\label{lem: convex bound}
	Set $\sigma = \mathrm{Re} (s)$. If $\varDelta $ is not a square, then
	\begin{equation}\label{3eq: trivial bound for L(s, Delta)}
	 L  (s, \varDelta) \Lt_{ s, \vepsilon } |\varDelta|^{\max \left\{0, \frac {1} 2 - \frac 1 2 \sigma, \frac {1} 2 - \sigma \right\} + \vepsilon },
	\end{equation}
	for any $\vepsilon > 0$, with the implied constant uniform for $s$ in compact sets.
\end{lemma}

\begin{proof}
Consider the last formula in \eqref{2eq: L(s, Delta) = ...}. 
By the trivial bound $
|L (s, \chi_D)| \leqslant \zeta (\sigma) $ for $\sigma > 1$, the functional equation for $L (s, \chi_D)$, and the Phragm\'en--Lindel\"of convex principle, we infer that $
L (s, \chi_D) \Lt_{s, \vepsilon} |D|^{\max \left\{0, \frac {1} 2 - \frac 1 2 \sigma, \frac {1} 2 - \sigma \right\} + \vepsilon }, $ whereas  the finite sum over $d$ in \eqref{2eq: L(s, Delta) = ...} is trivially bounded by $ O \big( q^{\max \left\{0, 1-2\sigma \right\} + \vepsilon} \big) $. Then \eqref{3eq: trivial bound for L(s, Delta)} follows immediately.
\end{proof}

	\section{Proof of Theorem \ref{thm-dimension-formula}}
\setcounter{equation}{0}

We consider the mean
\begin{align}\label{3eq: defn of P(s)}
	M^{\ssharp}(s) = 
	\sum_{f\in H^{\star}_{2k} (N^3)} \epsilon_f \frac{L(s, \sym^2(f)) }{L(1, \sym^2(f))},
\end{align}
so that $ |H^{\splus}_{2k} (N^3)| - |H^{\sminus}_{2k} (N^3)| =  
M^{\ssharp} (1) $.
We shall prove in \S \ref{sec: proof of thm P(s)} the following exact formula for $M^{\ssharp}(s)$.

\begin{thm}\label{thm: P(s) = ...}
Assume that $N > 1$ is squarefree and   $k > 1$.	For $2-2 k < \mathrm{Re} (s) < 2k -1 $, we have
\begin{align}\label{3eq: thm 3.1}
M^{\ssharp}(s) = M_0^{\ssharp} (s) + M_{1}^{\ssharp}(s) + M_{2}^{\ssharp} (s),
\end{align}
with
\begin{equation}\label{3eq: thm 3.1, M 0}
\begin{split}
M_0^{\ssharp} (s) =
\frac{(2k  -  1)\sqrt N \varphi (N) }{2\pi}  \frac{\Gamma\big(  k   -   \frac 1 2 s  \big)}
{\Gamma \big( k   +   \frac 1 2 s  \big)} \bigg(\frac{N^{3/2} } {2\pi } \bigg)^{\hskip -1pt 1-s}     { L (s, -4N)} ,
\end{split}
\end{equation}
\begin{equation}\label{3eq: thm 3.1, M 1}
\begin{split}
M_1^{\ssharp}(s) =  &\,
\frac{(2k   -   1)N^{3/2}   }{\pi}  \frac{\Gamma\big(k -   \frac 1 2 s \big)}
{\Gamma \big(k  + \frac 1 2 s\big)} \bigg( \frac {N^{3/2}} {2\pi  } \bigg)^{\hskip -1pt 1-s}  \sum_{d\mid N}  \frac{\mu(d)}{d }    \\
&   \cdot \sum_{1\leqslant n< 2d/N^{3/2} }  { L \big(s,(n N^2/d)^2-4N\big)} {_2F_1}\bigg(k-\frac{s}{2},1-k-\frac{s}{2}; \frac{1}{2};
\frac{n^2N^3}{4d^2 }\bigg)    ,
\end{split}
\end{equation}
and
	\begin{equation}\label{3eq: thm 3.1, M 2}
		\begin{split}
			M_2^{\ssharp}(s) =   &\,  \frac { (-1)^k (2\pi)^s   } { \pi^2 N^{3(k-1)} }  \cos \Big(\frac{\pi s}{2}  \Big)
			\frac{\Gamma (2k-s)}{\Gamma(2k-1)}   \sum_{d\mid N}\frac{\mu(d)}{d^{s-2k+1}}     \\
			&   \cdot   \sum_{n > 2d/N^{3/2} }   \frac { L \big(s,(n N^2/d)^2-4N\big)} {n^{2k-s}}
			{_2F_1}\bigg(k-\frac{s}{2},
			k+\frac{1}{2}-\frac{s}{2};
			2k; \frac{4d^2 } {n^2N^3} \bigg),
		\end{split}
	\end{equation}
where the   $n$-series converges absolutely to an analytic function of $s$ in the given region.
\end{thm}

Firstly, note that the inequality $1\leqslant n< 2d/N^{3/2}$  in \eqref{3eq: thm 3.1, M 1} has only two integer solutions (with $d|N$): either $N=d=2, n = 1$, or $N=d=3, n = 1$. Now let $s = 1$. The expressions in \eqref{3eq: thm 3.1, M 1} and \eqref{3eq: thm 3.1, M 2}  simplify greatly because
\begin{itemize} 
	\item [(1)] $\cos (\pi s/ 2)$ vanishes at $s = 1$, and
	\item [(2)] for $N = 2$ or $3$, the only remaining hypergeometric function may be evaluated explicitly by \eqref{2eq: Gauss evaluated}.
\end{itemize}
After some calculations, we obtain
\begin{equation}
	\begin{split}
	M^{\ssharp}(1) = \frac{ \sqrt {N}  \varphi (N)}{ \pi}      { L (1, -4N)} \, & - \delta (N, 2) \frac{2 \sqrt 2    }{\pi}  L  (1, -4) \cos \Big(\pi \frac {2k-1} 4 \Big) \\
	&  - \delta (N, 3) \frac{2 \sqrt3    }{\pi}  L  (1, -3) \cos \Big(\pi \frac {2k-1} 3\Big).
	\end{split}
\end{equation}
Here $2 \sqrt{2}$ and $2 \sqrt{3}$ come from $2 N^{3/2}/ d$ for $N = d = 2$ and $N=d=3$ respectively. 
Recall that $N > 1$ is squarefree. 
By \eqref{2eq: L(s, Delta) = ...} in Lemma \ref{lemma-Zagier}, along with Dirichlet's class number formula \eqref{1eq: Dirichlet class}, we have
\begin{equation*}
   L (1, -4N) =
	\left\{
	\begin{aligned}
		&\frac{\pi}{2\sqrt{N}}h(-4N),  && \text{ if
		} N\equiv 1, 2 \, (\mathrm{mod}\, 4),   \\
		& \frac{\pi}{2\sqrt{N}}h(-N)\left(3 - \chi_{-N}(2)\right), 
		   && \text{ if
		} N\equiv 3 \, (\mathrm{mod}\, 4), N\neq 3 ,\\
		&\frac{\pi}{6\sqrt{3}} \left(3 - \chi_{-3}(2)\right),    && \text{ if } N=3,
	\end{aligned}
	\right.
\end{equation*}
and
\begin{align*}
  L(1, -4) = \frac {\pi} {4 }, \quad \quad	 L(1, -3) = \frac {\pi} {3 \sqrt{3} },
\end{align*}
in which
\begin{equation*}
	\chi_{-N} (2) = \left\{ \begin{aligned}
		& 1, && \text{ if } N\equiv 7 \, (\mathrm{mod}\, 8), \\
		& - 1,  && \text{ if } N\equiv 3 \, (\mathrm{mod}\, 8).
	\end{aligned} \right.
\end{equation*}
Note that $\sigma_{-1} (2) = 3/2$, that  $h (-3) = h (-4) = h (-8) = 1$, and that $w (-3) = 6$, $w(-4) = 4$ and $w (D) = 2$ for every fundamental discriminant $D < -4$. Moreover,
\begin{equation*}
\cos \Big(\pi \frac {2k-1} 4\Big) = \left\{ \begin{aligned}
& 1/\sqrt 2 , && \text{ if } k \equiv 0, 1 \, (\mathrm{mod}\, 4), \\
& - 1/\sqrt 2 , && \text{ if } k \equiv 2, 3 \, (\mathrm{mod}\, 4),
\end{aligned} \right.
\end{equation*}
and
\begin{equation*}
	\cos \Big(\pi \frac {2k-1} 3\Big) = \left\{ \begin{aligned}
		& 1/2, && \text{ if } k \equiv 0, 1 \, (\mathrm{mod}\, 3), \\
		& - 1, && \text{ if } k \equiv 2 \, (\mathrm{mod}\, 3).
	\end{aligned} \right.
\end{equation*}
Combining the foregoing results, Theorem \ref{thm-dimension-formula} follows immediately.


\section{Proof of Theorem \ref{thm: P(s) = ...}}\label{sec: proof of thm P(s)}
\setcounter{equation}{0}

The main ideas in the proof of Theorem \ref{thm: P(s) = ...} below are essentially due to Balkanova, Frolenkov, and Shenhui Liu \cite{Ba-Fr-Mean,Liu-First-Moment}.

Recall the definition of $M^{\ssharp}(s)$ in \eqref{3eq: defn of P(s)}. We first assume that  $3/2 < \mathrm{Re}(s)< 2k-1$. In view of \eqref{eq-symmetric-square} and \eqref{2eq: defn Delta sharp}, we infer that
\begin{align*}
	M^{\ssharp}(s) = \zeta^{(N)} (2s)  \sum_{ (n,N)=1 }\frac{\Delta_{2k, N}^{\ssharp}(1, n^2 )}{n^s}.
\end{align*}
By the Petersson formula \eqref{2eq: Petersson sharp} in Proposition \ref{prop-sPTF}, we have
\begin{align*}
	M^{\ssharp}(s)=\frac{(2k-1)N^{3/2}}{\pi} \zeta^{(N)} (2s)
	\mathop{ \mathop{\sum\sum}_{c, \hskip 1pt n    }}_{(c n,N)=1}
	\frac{S\big(\overbar {N}{}^3 ,n^2;c\big)}{c n^s} J_{2k-1}\left(\frac{4\pi n}{N^{3/2}c}\right). 
\end{align*}
By   \eqref{2eq: Weil} and \eqref{2eq: bounds for Bessel}, the double series is absolutely convergent if $ \mathrm{Re}(s) > 3/2$. 

Next, we use the Mellin--Barnes formula for the Bessel function as in \eqref{2eq: Mellin-Barnes}, obtaining
\begin{equation*}
	M^{\ssharp}(s) =
	\frac{(2k-1)N^{3/2}}{\pi} \zeta^{(N)} (2s) \hskip -3 pt \sum_{(c, N) = 1} \hskip -2 pt \frac 1 c \cdot
	\frac{1}{4\pi i}\int_{(\sigma)} \hskip -1 pt \frac{\Gamma\big(k - \frac 1 2 + \frac 1 2 w \big)}
	{\Gamma\big(k + \frac 1 2   - \frac 1 2 w\big)} \bigg( \hskip -1 pt \frac{N^{3/2} c} {2\pi } \hskip -1 pt \bigg)^{\hskip -1 pt w}  \hskip - 4 pt \sum_{(n, N) = 1} \hskip - 4 pt
	\frac{S\big(\overbar {N}{}^3 ,n^2;c\big)}{  n^{s+w}}
	d w,
\end{equation*}
where,  by Stirling's formula and Weil's bound,  we require $  1 - \mathrm{Re} (s)   < \sigma < - 1/2$ to guarantee the absolute convergence of  the sums over $c$ and $n$ as well as the integral over $w$. Note that our assumption $\mathrm{Re} (s) < 2k-1$ implies $1-\mathrm{Re} (s) > 1-2k$, so  the condition $     \sigma > 1-2k $ for  \eqref{2eq: Mellin-Barnes} is guaranteed.  For the innermost sum over $n$, we   use the M\"obius function to relax the coprimality condition $(n,N)=1$  and write $n = a + c m$ so that the $m$-sum yields the Hurwitz zeta function $\zeta (s+w, a/c)$ as in \eqref{2eq: Hurwitz}. Thus
\begin{align*}
	\sum_{(n, N) = 1} \hskip -2pt
	\frac{S\big(\overbar {N}{}^3 ,n^2;c\big)}{  n^{s+w}}  = \frac 1 {c^{s+w}}  \sum_{d\mid N}\frac{\mu(d)}{d^{s+w}}
	\sum_{ a = 1}^c S \big(\overbar N{}^3 ,  (ad)^2;c\big)
	\zeta\Big(s+w,\frac{ a}{c}\Big).
\end{align*}
Recall that $\zeta (s+w, { a}/{c} )$ has a simple pole at $w = 1-s$ with residue $1$. We shift the $w$-contour to the left down to $\mathrm{Re} (w) = \sigma_1 $, with $1-2k < \sigma_1 < - \mathrm{Re} (s) $, crossing   a simple pole at $w = 1-s$.\footnote{We remark that in the case $k=1$ there is an issue with the proof of \cite[Lemma 5.1]{Ba-Fr-Mean}.  When they shift the integral contour to the left, the simple pole of the gamma function at $w = 1-2k = -1$ is also crossed. Additional work is required to address the problem of the analytic continuation of the resulting expression from $3/2 < \mathrm{Re} (s) < 2$ to $0 < \mathrm{Re} (s) < 1$. 
} A technical matter here is to check  that  as $ |\mathrm{Im}(w)| \rightarrow \infty$ the integrand $\rightarrow 0$ uniformly for $\sigma_1 \leqslant \mathrm{Re}(w) \leqslant \sigma$,  but this can be done for $\mathrm{Re} (s) > 3/2$ by the Stirling formula and the crude bound for the Hurwitz zeta function in \eqref{2eq: bounds for zeta}. 
	Therefore
\begin{equation}\label{4eq: P(s) = ... 2}
	\begin{split}
		M^{\ssharp}(s) = &\,
		\frac{(2k \hskip -0.5 pt - \hskip -0.5 pt 1)N^{3/2}}{2\pi} \zeta^{(N)} (2s) \frac{\Gamma \big(\hskip -0.5 pt k \hskip -0.5 pt - \hskip -0.5 pt  \frac 1 2 s \hskip -0.5 pt \big)}
		{\Gamma\big(\hskip -0.5 pt k \hskip -0.5 pt + \hskip -0.5 pt \frac 1 2 s \hskip -0.5 pt \big)} \bigg(\hskip -1 pt \frac{N^{3/2} } {2\pi } \hskip -1 pt \bigg)^{\hskip -1 pt 1-s} \hskip -1 pt \sum_{d\mid N} \hskip -1 pt \frac{\mu(d)}{d} \hskip -4 pt \sum_{ {(c,N)=1}} \hskip -2 pt \frac{1}{c^{1+s}}
		\hskip -1 pt \sum_{ a = 1}^c \hskip -1 pt S\big(\overbar N{}^3 , (ad)^2;c\big) \\
		& + \frac{(2k-1)N^{3/2}}{\pi} \zeta^{(N)} (2s) \sum_{d\mid N}\frac{\mu(d)}{d^{s}}  \hskip -2pt \sum_{(c, N) = 1}\frac 1 {c^{1+s}}   \\
		& \hskip 45 pt \cdot
		\frac{1}{4\pi i}\int_{(\sigma_1)} \frac{\Gamma\left(k - \frac 1 2 + \frac 1 2 w \right)}
		{\Gamma\left(k + \frac 1 2   - \frac 1 2 w\right)} \bigg( \hskip -1 pt \frac{N^{3/2} } {2\pi d } \hskip -1 pt \bigg)^{\hskip -1 pt w}    \sum_{ a = 1}^c S \big(\overbar N{}^3 ,  (ad)^2;c\big)
		\zeta\Big(s+w,\frac{ a}{c}\Big)
		d w  .
	\end{split}
\end{equation}
For the first term in \eqref{4eq: P(s) = ... 2}, with the observation that the $a$-sum may be rewritten as
\begin{align*}
	\sum_{a (\mathrm{mod}\, c)} S \big( N, (a \overbar N{}^2 d)^2 ; c \big) = \sum_{a (\mathrm{mod}\, c)} S  ( N, a^2 ; c  ),
\end{align*}
we obtain the term $M_0^{\ssharp} (s)$ in \eqref{3eq: thm 3.1, M 0} by virtue of Lemma \ref{Dirichlet-series-involving-Kloosterman-sum} (with $m=N$ and $n=0$) and Remark \ref{rem: L(N)(s) = L(s)}.
Applying Hurwitz's formula \eqref{2eq: Hurwitz formula}, we have 
\begin{align*}
&\quad \ 	\sum_{ a = 1}^c   S \big(\overbar N{}^3 ,  (ad)^2;c\big)
	\zeta\Big(s+w,\frac{ a}{c}\Big) \\
	& = 2 (2\pi)^{s+w-1} \Gamma (1-s-w) \sin\Big(\pi \frac{s+w}{2} \Big) \sum_{a (\mathrm{mod}\, c)}  S \big(\overbar N{}^3 ,  (ad)^2;c\big) F \Big(\frac a c, 1-s-w \Big).
\end{align*}
Substituting this into the second term in  \eqref{4eq: P(s) = ... 2} and opening the function $F  (  a / c, 1-s-w  )$ according to \eqref{2eq: defn of F(b, s)}, we arrive at
\begin{align*}
	\frac{(2k-1)N^{3/2}   }{\pi} \zeta^{(N)} (2s) \sum_{d\mid N}\frac{\mu(d)}{d^{s}} \sum_{n=1}^{\infty} \frac {I_{k, s} \big( n N^{3/2}  / d \big)} {(2\pi n)^{1-s}}  \hskip -2pt \sum_{(c, N) = 1}\frac 1 {c^{1+s}} \sum_{a (\mathrm{mod}\, c)}  S \big(\overbar N{}^3 ,  (ad)^2;c\big)    {e\Big(\frac {an} c\Big)},
\end{align*}
in which $I_{k, s} \big( n N^{3/2} / d \big)$ is the integral defined as in \eqref{I-z-x}; the absolute convergence of the sums over $c$ and $n$ may be easily verified for $3/2 < \mathrm{Re} (s) < 2k-1$ (recall that $1-2k < \sigma_1 < - \mathrm{Re} (s) $).   An application of Lemmas \ref{Lemma-Gauss-hyper} and \ref{Dirichlet-series-involving-Kloosterman-sum} (along with Remark \ref{rem: L(N)(s) = L(s)}) yields the sum of $M_1^{\ssharp} (s)$ and $M_2^{\ssharp} (s)$ as in \eqref{3eq: thm 3.1, M 1} and \eqref{3eq: thm 3.1, M 2}. 
Note that the $a$-sum above is equal to 
\begin{align*}
	\sum_{a (\mathrm{mod}\, c)} S \big( N, (a \overbar N{}^2 d)^2 ; c \big) {e\Big(\frac {an} c\Big)} = \sum_{a (\mathrm{mod}\, c)} S  ( N, a^2 ; c  ) e \bigg(\frac {a n N^2/d} {c}\bigg). 
\end{align*}
Moreover, since $N > 1$ is squarefree, it is easy to see that $\varDelta = (nN^2/d)^2 - 4 N = N (N (nN/d)^2 - 4) $ is divisible by $p$ but not $p^2 $ for any odd prime $p | N$, and that $\varDelta \equiv 8, 12 \, (\mathrm{mod}\, 16)$ when $2 | N$, and,  in view of Remark \ref{rem: L(N)(s) = L(s)}, it readily follows that  $L^{(N)} (s, \varDelta ) = L  (s, \varDelta  )$. 

We have thus established the validity of \eqref{3eq: thm 3.1} for $3/2 < \mathrm{Re} (s) < 2k-1$. Next, since $ (nN^2/d)^2 - 4 N $  could never be a square for any $n \geqslant 0$ (see the arguments at the end of the last paragraph), the functions $L  (s, -4 N)$ and $L \big(s,(n N^2/d)^2-4N\big)$ in \eqref{3eq: thm 3.1, M 0}--\eqref{3eq: thm 3.1, M 2} are   entire  according to Lemma \ref{lemma-Zagier}.
By \eqref{2eq: Gauss hypergeometric} and the crude estimate \eqref{3eq: trivial bound for L(s, Delta)} in Lemma \ref{lem: convex bound}, the infinite series over $n$ on the right-hand side of \eqref{3eq: thm 3.1, M 2} is absolutely and compactly convergent for $ 2-2k < \mathrm{Re} (s) < 2k-1 $ and hence gives rise to an analytic function of $s$ on this domain. Note that \eqref{2eq: Gauss hypergeometric} implies that the  hypergeometric function in \eqref{3eq: thm 3.1, M 2} is bounded when $n$ is large,  in a uniform way if $s$ is in a compact set.  On the other hand, the mean $ M^{\ssharp} (s) $ as defined in \eqref{3eq: defn of P(s)} is an entire function, since each $L(s, \sym^2(f))$ is entire.    
 Finally, the proof is completed by the principle of analytic continuation.


\newcommand{\etalchar}[1]{$^{#1}$}
\def\cprime{$'$}

\end{document}